\documentclass[11pt,leqno]{amsart}
\usepackage{hyperref}
\usepackage{epsfig}
\usepackage{enumitem}
\usepackage{amsmath, amssymb}
\usepackage{amscd}
\usepackage[all,cmtip,matrix,arrow]{xy}
\usepackage{mathrsfs}
\usepackage{graphicx}
\usepackage{tikz-cd}
\usepackage{mathabx}
\usepackage{array}
\usepackage{longtable}
\xyoption{rotate}
\xyoption{pdf}
\usepackage{xparse}

\makeatletter
\DeclareRobustCommand*{\bfseries}{%
  \not@math@alphabet\bfseries\mathbf
  \fontseries\bfdefault\selectfont
  \boldmath
}
\makeatother

\AtBeginEnvironment{tikzcd}{\setlength{\mathsurround}{0pt}}

\newtheorem{theo}{Theorem}[section]
\newtheorem{lemma}[theo]{Lemma}

\newtheorem{prop}[theo]{Proposition}

\newtheorem{cor}[theo]{Corollary}
\newtheorem{remark}[theo]{Remark}

\newtheorem{example}[theo]{Example}

\numberwithin{equation}{section}

\mathchardef\mhyphen="2D

\def\bL{\mathbb{L}}

\def\Z{\mathbb{Z}}

\def\wt{\widetilde}

\def\B{{\mathcal B}}

\def\bL{{\mathbf L}}

\def\pre-tr{\operatorname{pre-tr}}

\def\Hom{\operatorname{Hom}}

\newcommand{\tens}[1]{%
  \mathbin{\mathop{\otimes}\displaylimits_{#1}}%
}

\newcommand{\laxtimes}[1]{\mathop{\times\mkern-13mu\raise1.3ex\hbox{$\scriptscriptstyle\to$}_{#1}}}

\newcommand{\hy}{\mhyphen}

\newcommand{\indlim}[1][]{\mathop{\varinjlim}\limits_{#1}}
\newcommand{\inddlim}[1][]{{``{\indlim[#1]}"}}
\newcommand{\prolim}[1][]{\mathop{\varprojlim}\limits_{#1}}
\newcommand{\proolim}[1][]{{``{\prolim[#1]}"}}

\newcommand{\bbar}{\overline}

\newcommand{\xto}{\xrightarrow}

\newcommand{\hto}{\hookrightarrow}

\newcommand{\onto}{\twoheadrightarrow}

\newcommand{\lex}{\operatorname{lex}}
\newcommand{\rex}{\operatorname{rex}}

\newcommand{\Coh}{\operatorname{Coh}}

\newcommand{\Cone}{\operatorname{Cone}}
\newcommand{\Fiber}{\operatorname{Fiber}}

\newcommand{\cF}{{\mathcal F}}

\newcommand{\cL}{{\mathcal L}}

\newcommand{\cA}{{\mathcal A}}
\newcommand{\cB}{{\mathcal B}}

\newcommand{\cC}{{\mathcal C}}
\newcommand{\cE}{{\mathcal E}}

\newcommand{\cT}{{\mathcal T}}

\newcommand{\cH}{{\mathcal H}}

\newcommand{\la}{\langle}
\newcommand{\ra}{\rangle}

\newcommand{\incl}{\operatorname{incl}}

\newcommand{\Fun}{\operatorname{Fun}}

\newcommand{\Ab}{\operatorname{Ab}}

\newcommand{\Perf}{\operatorname{Perf}}

\newcommand{\Shv}{\operatorname{Shv}}

\newcommand{\coker}{\operatorname{coker}}

\newcommand{\im}{\operatorname{Im}}

\newcommand{\add}{\operatorname{add}}

\newcommand{\Ind}{\operatorname{Ind}}
\newcommand{\Pro}{\operatorname{Pro}}

\newcommand{\id}{\operatorname{id}}

\usepackage{amscd}

\title[Some remarks on Quillen's D\'evissage theorem]
{Some remarks on Quillen's D\'evissage theorem}

\author{Alexander I. Efimov}
\address{Steklov Mathematical Institute of RAS, Gubkin St. 8, GSP-1, Moscow 119991, Russia}
\email{efimov@mccme.ru}

\begin{document}

\begin{abstract} In this paper we give a different proof of Quillen's D\'evissage theorem using Barwick's theorem of the heart. The key ingredient is a certain short exact sequence of dg categories, which is closely related to the Auslander-type construction for nilpotent extensions which was used in papers by Kuznetsov and Lunts \cite{KL15}, Land and Tamme \cite{LT19}, and the author \cite{E20}.
\end{abstract}


\maketitle

\tableofcontents

\section{Introduction}
\label{sec:intro}

We recall the following classical result of Quillen.

\begin{theo}\label{th:Quillen} \cite[Theorem 4]{Qui}
Let $\cA$ be a small abelian category, and let $\cB\subset\cA$ be a strictly full abelian subcategory which is closed under taking subobjects and quotients. Suppose that each object of $\cA$ has a finite filtration with subquotients in $\cB.$ Then we have isomorphisms $K_n(\cB)\xto{\sim}K_n(\cA)$ for $n\geq 0.$
\end{theo}

Another important result which is due to Barwick is about $K$-theory of stable $\infty$-categories with bounded $t$-structures.

\begin{theo}\label{th:Barwick} \cite[Theorem 6.1]{Bar15}
Let $\cT$ be a small stable $\infty$-category with a bounded $t$-structure, i.e. $\cT$ is generated as a stable subcategory by the heart $\cT^{\heartsuit}.$ Then we have isomorphisms $K_n(\cT^{\heartsuit})\xto{\sim} K_n(\cT)$ for $n\geq 0.$
\end{theo}
 
Both results are extremely useful for computations of $K$-theory. The original proof of Theorem \ref{th:Quillen} in \cite{Qui} uses Quillen's Theorem A and the $Q$-construction, and the argument is quite mysterious from the categorical perspective. On the other hand, Barwick's proof of Theorem \ref{th:Barwick} uses the machinery of exact $\infty$-categories as defined in \cite{Bar15}, and the argument is much more direct. In this paper we show that in fact Theorem \ref{th:Barwick} implies Theorem \ref{th:Quillen}.

We now explain the main construction. Let $\cA\supset \cB$ be as in Theorem \ref{th:Quillen}. Suppose moreover that each object of $\cA$ is an extension of two objects of $\cB.$ The general case reduces to this case, see the proof of Theorem \ref{th:Quillen} in Section \ref{sec:proofs}.

Consider the category $\cE_{\cA,\cB}$ of pairs $(X,Y),$ where $X\in\cA$ and $Y\subset X$ is a subobject such that $Y\in\cB$ and $X/Y\in\cB.$ Then the category $\cE_{\cA,\cB}$ is quasi-abelian, in particular we can consider it as an exact category. We refer to Subsection \ref{ssec:quasi_abelian} for the definition of a quasi-abelian category. 

We prove the following result.

\begin{theo}\label{th:key_construction_intro}
With the above notation, the following holds.
\begin{enumerate}[label=(\roman*), ref=(\roman*)]
	\item We have a natural semi-orthogonal decomposition
	\begin{equation*}
		D^b(\cE_{\cA,\cB})=\la D^b(\cB), D^b(\cB)\ra.
	\end{equation*}
	We denote by $\Phi_1,\Phi_2:D^b(\cB)\to D^b(\cE_{\cA,\cB})$ the inclusion functors for the first and second components, respectively. We denote by $\Phi_1^L$ the left adjoint to $\Phi_1,$ and by $\Phi_2^R$ the right adjoint to $\Phi_2.$ \label{SOD}
	\item We have a short exact sequence of dg categories
	\begin{equation}\label{eq:key_ses}
		0\to \cT_{\cA,\cB}\xto{\Psi} D^b(\cE_{\cA,\cB})\xto{q} D^b(\cA)\to 0.
	\end{equation}
	Here the functor $q$ is induced by the (exact) forgetful functor $\cE_{\cA,\cB}\to \cA.$ The functor $\B\to \cT_{\cA,\cB},$ $X\mapsto\Cone((X,0)\to (X,X)),$ is fully faithful, and its essential image is the heart of a bounded $t$-structure on $\cT_{\cA,\cB}.$ \label{ses}
	\item Both compositions $q\circ \Phi_i,$ $i=1,2,$ are isomorphic to the derived functor of the inclusion $\cB\hto\cA.$ The compositions $\Phi_1^L[-1]\circ \Psi$ and $\Phi_2^R\circ \Psi$ are left inverses to the realization functor $D^b(\cB)\to \cT_{\cA,\cB}.$ \label{functors_on_components}
\end{enumerate}
\end{theo}

Theorem \ref{th:Quillen} follows almost immediately from Theorems \ref{th:key_construction_intro} and \ref{th:Barwick}. 
  
Part \ref{SOD} of Theorem \ref{th:key_construction_intro} is straightforward. To prove \ref{ses} and \ref{functors_on_components} we give an explicit description of the category $\cL\cH(\cE_{\cA,\cB})$ -- the right abelian envelope of the category $\cE_{\cA,\cB},$ see Section \ref{sec:description_of_the_envelope}. Recall from \cite{Sch99} that for a quasi-abelian category $\cE$ its right abelian envelope is the heart of the bounded $t$-structure on $D^b(\cE)$ whose connective part is the subcategory of complexes in non-negative homological degrees. This heart consists of complexes of the form $\Cone(X\xto{f} Y),$ where $f$ is a monomorphism (not necessarily strict), see Proposition \ref{prop:generalities_on_abelian_envelope}. Both functors in \eqref{eq:key_ses} are $t$-exact, where we take the standard $t$-structure on $D^b(\cA).$

The category $D^b(\cE_{\cA,\cB})$ was in fact considered previously (with a different description) in \cite{KL15, E20} in the case $\cA=\Coh(A),$ $\cB=\Coh(A/I),$ where $A$ is a commutative noetherian ring and $I\subset A$ is an ideal such that $I^2=0.$ This category is also closely related to a certain lax pullback which was implicitly used in \cite[Proofs of Theorem 2.25 and Corollary 3.5]{LT19}, see Example \ref{ex:familiar_example} for details.

{\noindent{\bf Convention.}} We will use only the homological grading.

{\noindent{\bf Acknowledgements.}} I am grateful to an anonymous referee for helpful remarks and suggestions.  
This work was performed at the Steklov International Mathematical Center
and supported by the Ministry of Science and Higher Education of the
Russian Federation (agreement no. 075-15-2025-303).

\section{Preliminaries}

\subsection{Quasi-abelian categories}
\label{ssec:quasi_abelian}

We refer to \cite{Sch99} for the general theory of quasi-abelian categories. A closely related notion is that of a torsion theory in an abelian category as defined in \cite{HRS96}. This relation is studied in \cite[Section 5.4 and Appendix B]{BVdB03}. 

Recall that an additive category $\cE$ is called quasi-abelian if $\cE$ has kernels and cokernels, the class of strict monomorphisms is closed under pushouts and the class of strict epimorphisms is closed under pullbacks. Here a monomorphism $f$ is strict if $f=\ker(\coker(f)),$ and similarly for epimorphisms. In particular, a quasi-abelian category $\cE$ has an exact structure, for which the admissible short exact sequences are of the form
\begin{equation}\label{eq:adm_ses}
0\to X\xto{f} Y\xto{g} Z\to 0,
\end{equation}
where $f=\ker(g)$ and $g=\coker(f).$ In particular, we have the bounded derived category $D^b(\cE),$ which we consider as a dg category over $\Z.$

Clearly, any abelian category is quasi-abelian. An additive functor between quasi-abelian categories is called exact if it preserves admissible short exact sequences.
A functor $F:\cE\to\cA$ from a quasi-abelian category $\cE$ to an abelian category $\cA$ is called right exact if for any admissible short exact sequence \eqref{eq:adm_ses} we have an exact sequence
\begin{equation*}
F(X)\to F(Y)\to F(Z)\to 0.
\end{equation*}

\begin{prop}\label{prop:generalities_on_abelian_envelope}\cite[Definition 1.2.18, Propositions 1.2.32 and 1.2.34]{Sch99} Let $\cE$ be a small quasi-abelian category. 
\begin{enumerate}[label=(\roman*),ref=(\roman*)]
	\item The category $D^b(\cE)$ has a bounded $t$-structure, whose heart consists of objects of the form $\Cone(X\xto{f} Y),$ where $f$ is a monomorphism (but not necessarily a strict monomorphism). This heart is denoted by $\cL\cH(\cE).$ The inclusion functor $\cE\to\cL\cH(\cE),$ $X\mapsto X,$ is exact.
	\item The functor $D^b(\cE)\to D^b(\cL\cH(\cE))$ is an equivalence. \label{D^b_is_equiv_toD^b_of_the_envelope}
	\item The functor $\cE\to \cL\cH(\cE)$ satisfies the following universal property: for any small abelian category $\cA$ it induces an equivalence
	\begin{equation*}
	\Fun^{\rex}(\cL\cH(\cE),\cA)\xto{\sim}\Fun^{\rex}(\cE,\cA).
	\end{equation*}
	Here the superscript $\rex$ means ``right exact''.
\end{enumerate}
\end{prop}

Following \cite[Section 1.2.4]{Sch99}, we call the category $\cL\cH(\cE)$ the (right) abelian envelope of $\cE.$ We recall the following characterization of the category $\cL\cH(\cE).$

\begin{prop}\label{prop:characterization_of_LH}\cite[Proposition 1.2.36]{Sch99}
	Let  $\cE$ be a small quasi-abelian category and let $\cA$ be a small abelian category. Suppose that we have a fully faithful functor $F:\cE\to\cA$ such that the following conditions hold.
	\begin{enumerate}[label=(\roman*),ref=(\roman*)]
		\item The essential image $F(\cE)\subset\cA$ is closed under taking subobjects. 
		\item For every object $X\in\cA$ there exists an object $Y\in\cE$ and an epimorphism $F(Y)\onto X.$
	\end{enumerate}
	Then $F$ extends to an equivalence of categories $\cL\cH(\cE)\xto{\sim}\cA.$
\end{prop}

For completeness we also recall the relation between quasi-abelian categories and torsion pairs in abelian categories. Recall from \cite{HRS96} that for an abelian category $\cA$ a pair of strictly full subcategories $(\cT,\cF)$ is called a torsion pair if $\Hom(\cT,\cF)=0$ and for any $X\in\cA$ there is a short exact sequence
\begin{equation*}
0\to Y\to X\to Z\to 0,\quad Y\in\cT,\,Z\in\cF.
\end{equation*}
For such a pair $(\cT,\cF)$ we refer to the objects of $\cT$ resp. $\cF$ as torsion resp. torsion-free objects. Both $\cT$ and $\cF$ are additive subcategories closed under extensions, $\cT$ is closed under taking quotients and $\cF$ is closed under taking subobjects. Moreover, both $\cT$ and $\cF$ are quasi-abelian and the inclusion functors $\cT\to\cA$ and $\cF\to\cA$ are exact. Such a torsion pair is called cotilting if every object of $\cA$ is a quotient of an object of $\cF.$

\begin{prop}\cite[Proposition B.3]{BVdB03}
Let $\cE$ be a small additive category. The following are equivalent.
\begin{enumerate}[label=(\roman*),ref=(\roman*)]
	\item $\cE$ is quasi-abelian.
	\item There exists a cotilting torsion pair $(\cT,\cF)$ in a small abelian category $\cA$ with an equivalence $\cE\simeq \cF.$ \label{via_cotilting_pair}
\end{enumerate}
In the situation of \ref{via_cotilting_pair} we have $\cA\simeq\cL\cH(\cE).$
\end{prop} 

\subsection{Ind-objects and pro-objects}
\label{ssec:ind_obj_pro_obj}

We refer to \cite{KS01} for the general theory of ind-objects in ordinary categories. For a small abelian category $\cA$ we have the abelian category $\Ind(\cA)$ of ind-objects, and the inclusion functor $\cA\to\Ind(\cA)$ is exact. We write the objects of $\Ind(\cA)$ in the form $\inddlim[i\in I]X_i,$ where $I$ is a directed poset and we have a functor $I\to\cA,$ $i\mapsto X_i.$ Recall that we have an equivalence
\begin{equation*}
\Ind(\cA)\simeq \Fun^{\lex}(\cA^{op},\Ab),\quad \inddlim[i] X_i\mapsto \indlim[i] h_{X_i},
\end{equation*}
where $h_{X_i}=\Hom(-,X_i).$ Here the superscript $\lex$ means ``left exact''. 

It is convenient to give another equivalent description of the category $\Ind(\cA)$ via sheaves. Namely, consider the Grothendieck topology on the category $\cA,$ for which a sieve $S\subset h_X$ is a covering sieve if it contains an epimorphism $Y\onto X.$ We will call it the single-epi topology, following \cite{Ka22}. Then a functor $F:\cA^{op}\to \Ab$ is left exact if and only if $F$ is an additive sheaf for the single-epi topology. Hence, we have an equivalence
\begin{equation*}
\Ind(\cA)\simeq \Shv^{\add}(\cA;\Ab).
\end{equation*}

Given a functor $F:\cA^{op}\to\Ab,$ considered as a presheaf of abelian groups, we denote by $F^{\sharp}$ its sheafification for the single-epi topology.  We will use the following standard statement.

\begin{prop}\label{prop:sheafification_of_additive_presheaves}
For an additive functor $F:\cA^{op}\to \Ab,$ the sheafification $F^{\sharp}$ is an additive sheaf. In particular, the inclusion $\Ind(\cA)\to \Fun^{\add}(\cA^{op},\Ab)$ has an exact left adjoint, given by the composition
\begin{equation}\label{eq:left_adjoint_via_sheafification}
\Fun^{\add}(\cA^{op},\Ab)\xto{(-)^{\sharp}} \Shv^{\add}(\cA;\Ab)\simeq \Ind(\cA).
\end{equation} 
\end{prop}

\begin{proof}
We only need the following observation: for any objects $X,Y\in\cA$ and an epimorphism $(f,g):Z\onto X\oplus Y,$ we can find a pair of epimorphisms $f':X'\onto X$ and $g':Y'\onto Y$ such that the morphism $(f',g'):X'\oplus Y'\to X\oplus Y$ factors through $Z.$ Namely, it suffices to take $X'=\ker(g),$ $Y'=\ker(f),$ and define $f'$ and $g'$ as the compositions $X'\to Z\xto{f}X$ and $Y'\to Z\xto{g}Y.$

Exactness of the functor \eqref{eq:left_adjoint_via_sheafification} is automatic since the sheafification is always exact.
\end{proof}

Note that exactness of the functor \eqref{eq:left_adjoint_via_sheafification} is also a trivial special case of the Gabriel-Kuhn-Popesco theorem \cite{GP64, Ku94}.

We also have the abelian category of pro-objects $\Pro(\cA)\simeq \Ind(\cA^{op})^{op}.$ We write the objects of $\Pro(\cA)$ in the form $\proolim[i\in I]X_i,$ where $I$ is a codirected poset and we are assuming a functor $I\to\cA,$ $i\mapsto X_i.$ We will need the following immediate application of Proposition \ref{prop:sheafification_of_additive_presheaves}.

\begin{cor}\label{cor:making_a_functor_right_exact}
Let $\cA$ and $\cB$ be small abelian categories. Then the inclusion functor
\begin{equation*}
\incl:\Fun^{\rex}(\cA,\Pro(\cB))\to \Fun^{\add}(\cA,\Pro(\cB))
\end{equation*}
has a right adjoint $\incl^R.$ Given an additive functor $F:\cA\to\Pro(\cB)$ and an object $X\in\cB,$ the morphism of functors
\begin{equation*}
\Hom(F(-),X)\to \Hom(\incl^R(F)(-),X)
\end{equation*}
exhibits the target as a sheafification of the source for the single-epi topology on $\cA.$ 
\end{cor} 

\begin{proof}
Indeed, we have a commutative square
\begin{equation*}
\begin{CD}
\Fun^{\rex}(\cA,\Pro(\cB)) @>{\sim}>> \Fun^{\lex}(\cB,\Shv^{\add}(\cA;\Ab))^{op}\\
@V{\incl}VV @VVV\\
\Fun^{\add}(\cA,\Pro(\cB)) @>{\sim}>> \Fun^{\lex}(\cB,\Fun^{\add}(\cA^{op},\Ab))^{op}.
\end{CD}
\end{equation*}
Here the horizontal functors are given by $F\mapsto (X\mapsto \Hom(F(-),X)).$ The corollary now follows directly from Proposition \ref{prop:sheafification_of_additive_presheaves}.
\end{proof}

\section{Right abelian envelope of the category $\cE_{\cA,\cB}$}
\label{sec:description_of_the_envelope}

Throughout this section, $\cA\supset\cB$ are small abelian categories as in Theorem \ref{th:key_construction_intro}. That is, $\cB$ is a strictly full subcategory of $\cA,$ which is closed under taking subobjects and quotients, and each object of $\cA$ is an extension of two objects of $\cB.$ The category $\cE_{\cA,\cB}$ is defined as in the introduction, i.e. it is the category of pairs $(X,Y),$ where $X\in\cA,$ $Y\subset X$ and $Y\in\cB,$ $X/Y\in\cB.$ We will describe directly the right abelian envelope $\cL\cH(\cE_{\cA,\cB})$ and deduce some of its structural properties.

We start with the following trivial observation.

\begin{lemma}\label{lem:intersections}
Let $X\in \cA$ be an object and suppose that $Y_1,Y_2\subset X$ are subobjects such that $X/Y_1, X/Y_2\in \cB.$ Then $X/(Y_1\cap Y_2)\in\cB.$
\end{lemma}

\begin{proof}
Indeed, this follows from the monomorphism $X/(Y_1\cap Y_2)\to X/Y_1\oplus X/Y_2.$
\end{proof}

We denote by $i:\cB\to\cA$ the inclusion functor and, abusing notation, we denote by the same symbol the induced functor on the pro-completions $i:\Pro(\cB)\to\Pro(\cA)$.
The latter functor has a left adjoint $i^L:\Pro(\cA)\to\Pro(\cB),$ which commutes with cofiltered limits and, on objects of $\cA,$ is given by
\begin{equation*}
i^L(X) = \proolim[\substack{Y\subset X,\\ Y\in\cB, X/Y\in\cB}] X/Y,\quad X\in\cA.
\end{equation*}
Consider now the functor
\begin{equation}\label{eq:direct_definition_of_j}
j:\cA\to\Pro(\cB),\quad j(X) = \proolim[\substack{Y\subset X,\\ Y\in\cB, X/Y\in\cB}] Y.
\end{equation}
Equivalently, $j$ can be defined by the functorial short exact sequence
\begin{equation}\label{eq:ses_defining_j}
0\to i(j(X))\to X\to i(i^L(X))\to 0.
\end{equation}
The following proposition is immediate from \eqref{eq:ses_defining_j} and also follows directly from \eqref{eq:direct_definition_of_j}. 

\begin{prop}\label{prop:j_preserves_epi}
The functor $j:\cA\to\Pro(\cB)$ preserves epimorphisms.
\end{prop} 

\begin{proof}
Consider a short exact sequence
\begin{equation*}
0\to X\to Y\to Z\to 0
\end{equation*}
in $\cA.$ Applying \eqref{eq:ses_defining_j} to $Y$ and $Z,$ we obtain an exact sequence
\begin{equation*}
X\to \ker(i(i^L(Y))\to i(i^L(Z)))\to \coker(i(j(Y))\to i(j(Z)))\to 0.
\end{equation*}
The first morphism is a composition of epimorphisms
\begin{equation*}
X\to i(i^L(X))\to \ker(i(i^L(Y))\to i(i^L(Z))),
\end{equation*} 
hence it is also an epimorphism. This proves that $j(Y)\to j(Z)$ is an epimorphism.
\end{proof}

Consider the inclusion functor $\incl:\Fun^{\rex}(\cA,\Pro(\cB))\to \Fun^{\add}(\cA,\Pro(\cB)),$ and denote by $\incl^R$ its right adjoint, which is described in Corollary \ref{cor:making_a_functor_right_exact}. We define the right exact functor $\wt{j}:\cA\to \Pro(\cB)$ by $\wt{j}=\incl^R(j).$

\begin{remark}
Another description of the functor $\wt{j}$ is the following. Consider the left derived functor $\bL i^L:D^-(\Pro(\cA))\to D^-(\Pro(\cB)),$ which exists since $\Pro(\cA)$ has enough projective objects. Then for $X\in\cA$ we have an isomorphism
\begin{equation*}
i(\wt{j}(X))\cong H_0(\Fiber(X\to i(\bL i^L(X)))).
\end{equation*}
However, the former description of $\wt{j}$ will be more useful for us.
\end{remark}

We will need the following observation.

\begin{prop}\label{prop:j_tilde_to_j_is_epi}
For $X\in\cA,$ the natural morphism $\wt{j}(X)\to j(X)$ is an epimorphism.
\end{prop}

\begin{proof}
Take some object $Y\in \cB.$ By Proposition \ref{prop:j_preserves_epi}, the functor $\Hom(j(-),Y)$ is a separated presheaf on $\cA$ with respect to the single-epi topology. Its sheafification is given by $\Hom(\wt{j}(-),Y)$ by Corollary \ref{cor:making_a_functor_right_exact}. Hence, we have a monomorphism of presheaves $\Hom(j(-),Y)\to \Hom(\wt{j}(-),Y).$ This exactly means that the morphism $\wt{j}(X)\to j(X)$ is an epimorphism for each $X\in\cA.$
\end{proof}

Now we define an abelian category, which will be in fact equivalent to the right abelian envelope $\cL\cH(\cE_{\cA,\cB}).$

{\noindent{\bf Construction.}} {\it We define the category $\cC_{\cA,\cB}$ as follows. Its objects are quadruples $(X,Y,u,v),$ where $X\in\cA,$ $Y\in\cB,$ $u:i(Y)\to X,$ $v:\wt{j}(X)\to Y,$ such that the following conditions hold.
\begin{enumerate}[label=(\roman*),ref=(\roman*)]
	\item The composition $\wt{j}(i(Y))\xto{\wt{j}(u)}\wt{j}(X)\xto{v}Y$ is zero.
	\item The following square commutes:
	\begin{equation}\label{eq:comm_square_for_u_and_v}
		\begin{CD}
			i(\wt{j}(X)) @>{i(v)}>> i(Y)\\
			@VVV @V{u}VV\\
			i(j(X)) @>>> X.
		\end{CD}
	\end{equation}
	Here the left vertical and the lower horizontal maps are the canonical ones. 
\end{enumerate}
The morphisms in $\cC_{\cA,\cB}$ are defined in the natural way.}

To illustrate the category $\cC_{\cA,\cB},$ we mention the following special case.

\begin{example}\label{ex:familiar_example}
Let $A$ be a right noetherian associative unital ring with a two-sided ideal $I\subset A$ such that $I^2=0.$ Consider the abelian category $\cA=\Coh\hy A$ of finitely generated right $A$-modules, and the full subcategory $\cB=\Coh\hy A/I\subset \cA.$ Then $\cA$ and $\cB$ satisfy the above conditions. The functors $j,\wt{j}:\cA\to\Pro(\cB)$ take values in the constant pro-objects and are given by
\begin{equation*}
j(M)=MI,\quad \wt{j}(M)=M\tens{A}I.
\end{equation*}
We have $\cC_{\cA,\cB}\simeq \Coh\hy D,$ where the (right noetherian) ring $D$ can be depicted as follows:
\begin{equation*}
D=\begin{pmatrix}
	A & I\\
	A/I & A/I
\end{pmatrix}.
\end{equation*}
This means that $D$ is the direct sum of matrix entries, and $D$ is equipped with the matrix multiplication. Explicitly, an object $(M,N,u,v)\in\cC_{\cA,\cB}$ corresponds to the right module $M\oplus N$ over $D,$ where the action of $D$ is given by the maps 
\begin{equation*}
M\otimes A\to M,\quad N\otimes A/I\to N,\quad N\otimes A/I\to i(N)\xto{u} M,\quad M\otimes I\to M\tens{A} I\xto{v} N.
\end{equation*} 

The categories $D^b(\Coh\hy D)$ and $\Perf(D)$ were used respectively in \cite[Section 8.1]{E20} and \cite[Section 5]{KL15} in the case when $A$ is commutative (this is a special case of a construction for a general nilpotent ideal).

The category $\Perf(D)$ was also used implicitly in \cite[Proofs of Theorem 2.25 and Corollary 3.5]{LT19}, when $A$ is not necessarily right noetherian. Namely, using the notation of loc. cit., this category is equivalent to the lax pullback:
\begin{equation}\label{eq:equiv_with_lax_pullback}
\Perf(D)\xto{\sim}\Perf(C(I,A))\laxtimes{\Perf(C(I,A/I))} \Perf(A/I).
\end{equation}
Here $C(I,A)$ and $C(I,A/I)$ are dg rings given by
\begin{equation*}
C(I,A)=\Cone(I\hto A),\quad C(I,A/I)=\Cone(I\xto{0}A/I),
\end{equation*}
and the map $C(I,A)\to C(I,A/I)$ is given by the identity on $I$ and by the projection $A\to A/I$ on $A.$ The equivalence \eqref{eq:equiv_with_lax_pullback} is given on the generator by 
\begin{equation*}
D\mapsto (C(I,A),A/I,\id)\oplus (C(I,A),0,0).
\end{equation*}
It is straightforward to check that the endomorphism dg algebra of the latter object is indeed quasi-isomorphic to $D.$
\end{example}

We first summarize the basic properties of the category $\cC_{\cA,\cB}.$

\begin{prop}\label{prop:properties_of_C_AB}
\begin{enumerate}[label=(\roman*),ref=(\roman*)]
	\item The category $\cC_{\cA,\cB}$ is abelian. The forgetful functors 
	\begin{equation*}
	\cC_{\cA,\cB}\to\cA,\,(X,Y,u,v)\mapsto X,\quad\text{and}\quad \cC_{\cA,\cB}\to\cB,\,(X,Y,u,v)\mapsto Y,
	\end{equation*} are exact. \label{abelian_category}
	\item For any object $(X,Y,u,v)\in\cC_{\cA,\cB},$ the object $X/u(i(Y))$ is contained in $\cB.$ \label{cokernel_in_B}
	\item Consider an object $(X,Y)\in\cE_{\cA,\cB},$ and let $u:i(Y)\to X$ be the inclusion morphism. Then there exists a unique morphism $v:\wt{j}(X)\to Y$ such that the quadruple $(X,Y,u,v)$ is a well-defined object of $\cC_{\cA,\cB}.$ More precisely, $v$ is given by the composition of natural maps:
	\begin{equation}\label{eq:composition_defining_v}
	\wt{j}(X)\to j(X)=\proolim[\substack{Z\subset X,\\ Z\in\cB, X/Z\in\cB}] Z\to Y
	\end{equation}
	 This gives a fully faithful functor $\alpha:\cE_{\cA,\cB}\to \cC_{\cA,\cB}.$ Its essential image consists of objects $(X',Y',u',v')$ such that $u'$ is a monomorphism. \label{inclusion_of_E_AB}
	 \item We have a well-defined fully faithful functor $\beta:\cB\to\cC_{\cA,\cB},$ given by $\beta(X)=(0,X,0,0).$ We have a torsion pair in $\cC_{\cA,\cB},$ given by $(\beta(\cB),\alpha(\cE_{\cA,\cB})).$ \label{torsion_pair}
\end{enumerate}
\end{prop}

\begin{proof} \ref{abelian_category} follows directly from the right exactness of the functors $i:\cB\to\cA$ and $\wt{j}:\cA\to\Pro(\cB).$

To prove \ref{cokernel_in_B}, denote by $Y'\subset X$ the image of $u:i(Y)\to X,$ in particular $Y'\in \cB.$ In the commutative square \eqref{eq:comm_square_for_u_and_v} the left vertical arrow is an epimorphism by Proposition \ref{prop:j_tilde_to_j_is_epi}, hence the canonical morphism $i(j(X))\to X$ factors through $Y'.$ This exactly means that $X/Y'\in\cB,$ as required. 

We prove \ref{inclusion_of_E_AB}. Suppose that $v:\wt{j}(X)\to Y$ is some morphism satisfying the required conditions. As above, in the commutative square \eqref{eq:comm_square_for_u_and_v} the left vertical arrow is an epimorphism. Also, the right vertical arrow is a monomorphism by assumption. Hence, there exists a unique diagonal arrow $i(j(X))\to i(Y),$ making both triangles commutative. This shows that $v$ coincides with the composition \eqref{eq:composition_defining_v}. On the other hand, it is clear that this composition satisfies the required conditions (note that $j(i(Y))=0$). Full faithfulness of the resulting functor $\alpha:\cE_{\cA,\cB}\to \cC_{\cA,\cB}$ follows directly. The description of the essential image of $\alpha$ follows using \ref{cokernel_in_B}.

It remains to prove \ref{torsion_pair}. The full faithfulness of $\beta$ is clear, and so is the vanishing of $\Hom(\beta(Z),\alpha(X,Y))$ for $Z\in\cB,$ $(X,Y)\in\cE_{\cA,\cB}.$ The right adjoint $\beta^R$ to $\beta$ is given by $i(\beta^R(X,Y,u,v))=\ker(u).$ The adjunction counit  $\beta(\beta^R(X,Y,u,v))\to (X,Y,u,v)$ is a monomorphism, and its cokernel is contained in the essential image of $\alpha$ by \ref{inclusion_of_E_AB}. This proves \ref{torsion_pair}.
\end{proof}

We now describe the right abelian envelope of $\cE_{\cA,\cB}.$

\begin{prop}\label{prop:abelian_envelope}
The fully faithful exact inclusion $\alpha:\cE_{\cA,\cB}\to \cC_{\cA,\cB}$ induces an equivalence $\cL\cH(\cE_{\cA,\cB})\simeq \cC_{\cA,\cB}.$ 
\end{prop}

\begin{proof}
By Proposition \ref{prop:properties_of_C_AB} \ref{torsion_pair} the category $\alpha(\cE_{\cA,\cB})$ is the full subcategory of torsion-free objects for a torsion pair in $\cC_{\cA,\cB}.$ Hence, by Proposition \ref{prop:characterization_of_LH} we only need to prove that any object of $\cC_{\cA,\cB}$ is a quotient of an object of $\alpha(\cE_{\cA,\cB}).$

Take some object $(X,Y,u,v)\in\cC_{\cA,\cB}.$ As above, put $Y'=\im(u:i(Y)\to X),$ so $Y'$ is an object of $\cB$ and the morphism $i(j(X))\to X$ factors through $Y'.$  Recall that the functor $\Hom(\wt{j}(-),Y)$ resp. $\Hom(\wt{j}(-),Y')$ is the sheafification of the separated presheaf $\Hom(j(-),Y)$ resp. $\Hom(j(-),Y')$ with respect to the single-epi topology on $\cA.$ It follows that there exists an epimorphism $f:Z\to X$ and a morphism $w:j(Z)\to Y$ such that the following diagram commutes:
\begin{equation}\label{eq:two_comm_squares}
\begin{tikzcd}
\wt{j}(Z) \ar[d, "\wt{j}(f)"] \ar[r] & j(Z) \ar[d, "w"] \ar[r, "j(f)"] & j(X)\ar[d] \\
\wt{j}(X)\ar[r, "v"]   & Y \ar[r] & Y'
\end{tikzcd}
\end{equation}
Now the morphism $w:j(Z)\to Y$ is represented by some morphism $g:Z'\to Y,$ where $Z'\subset Z$ and $Z',Z/Z'\in\cB.$ Using the commutativity of \eqref{eq:two_comm_squares} and replacing $Z'$ by a smaller subobject if necessary (so that $Z/Z'$ is still in $\cB$), we may assume that the following square commutes:
\begin{equation*}
\begin{CD}
i(Z') @>{i(g)}>> i(Y)\\
@VVV @V{u}VV\\
Z @>{f}>> X.
\end{CD}
\end{equation*}
By construction, we obtain a well-defined morphism $(f,g):\alpha(Z,Z')\to (X,Y,u,v).$ We also have a natural morphism $\alpha(Y,Y)\to (X,Y,u,v),$ and together these two morphisms define an epimorphism
\begin{equation*}
\alpha(Z,Z')\oplus\alpha(Y,Y)\to (X,Y,u,v).
\end{equation*}
This proves that the object $(X,Y,u,v)$ is a quotient of an object of $\alpha(\cE_{\cA,\cB}),$ as required. 
\end{proof}

Next, we observe that the category $\cA$ is naturally a Serre quotient of $\cC_{\cA,\cB}.$

\begin{prop}\label{prop:Serre_quotient}
The full subcategory $\beta(\cB)\subset \cC_{\cA,\cB}$ is a Serre subcategory, i.e. it is closed under taking subobjects, quotients and extensions. The functor $\pi:\cC_{\cA,\cB}\to \cA,$ $\pi(X,Y,u,v)=X,$ induces an equivalence $\bbar{\pi}:\cC_{\cA,\cB}/\beta(\cB)\xto{\sim} \cA.$
\end{prop}

\begin{proof}
The first statement is evident. Next, we see that the functor $\pi$ is essentially surjective: for any $X\in\cA$ we can choose $Y\subset X$ such that $Y,X/Y\in\cB,$ and then $X\cong \pi(\alpha(X,Y)).$

It remains to prove the full faithfulness of $\bbar{\pi}.$ By Proposition \ref{prop:properties_of_C_AB} \ref{torsion_pair}, each object of $\cC_{\cA,\cB}/\beta(\cB)$ is isomorphic to the image of an object of $\alpha(\cE_{\cA,\cB}).$ Take some objects $(X_1,Y_1),(X_2,Y_2)\in\cE_{\cA,\cB}.$ If $Z\subset\alpha(X_1,Y_1)$ is a subobject such that $\alpha(X_1,Y_1)/Z\in\beta(\cB),$ then $Z$ must be of the form $\alpha(X_1,Y_1')$ for some $Y_1'\subset Y_1$ such that $X_1/Y_1'\in\cB.$ On the other hand, any non-zero subobject of $\alpha(X_2,Y_2)$ is not contained in $\beta(\cB).$ We obtain
\begin{multline*}
\Hom_{\cC_{\cA,\cB}/\beta(\cB)}(\alpha(X_1,Y_1),\alpha(X_2,Y_2))=\indlim[\substack{Y_1'\subset Y_1,\\ X_1/Y_1'\in\cB}]\Hom_{\cC_{\cA,\cB}}(\alpha(X_1,Y_1'),\alpha(X_2,Y_2))\\
\xto{\sim} \Hom_{\cA}(X_1,X_2).
\end{multline*}
Here the second isomorphism follows from the following: for any $f:X_1\to X_2$ there exists $Y_1'\subset Y_1$ such that $X_1/Y_1'\in\cB$ and $f(Y_1')\subset Y_2.$ Namely, by Lemma \ref{lem:intersections} we can take $Y_1' = Y_1\cap f^{-1}(Y_2).$ This proves the proposition. 
\end{proof}

\section{Proofs of Theorem \ref{th:key_construction_intro} and \ref{th:Quillen}}
\label{sec:proofs}

\begin{proof}[Proof of Theorem \ref{th:key_construction_intro}]
We prove \ref{SOD}. We first consider the exact functors between quasi-abelian categories
\begin{equation*}
\Phi_1:\cB\to\cE_{\cA,\cB},\,\Phi_1(X)=(X,0),\quad \Phi_2:\cB\to\cE_{\cA,\cB},\,\Phi_2(X)=(X,X).
\end{equation*} 
The left adjoint $\Phi_1^L$ to $\Phi_1$ and the right adjoint $\Phi_2^R$ to $\Phi_2$ are also exact functors, given by
\begin{equation*}
\Phi_1^L(X,Y)=X/Y,\quad \Phi_2^R(X,Y)=Y.
\end{equation*}
Hence, we have a functorial admissible short exact sequence
\begin{equation*}
0\to \Phi_2(\Phi_2^R(X,Y))\to (X,Y)\to \Phi_1(\Phi_1^L(X,Y))\to 0,\quad (X,Y)\in\cE_{\cA,\cB}.
\end{equation*}
We also have the vanishing $\Phi_2^R\circ \Phi_1=0$ and the isomorphisms $\Phi_1^L\circ\Phi_1\xto{\sim}\id,$ $\id\xto{\sim}\Phi_2^R\circ\Phi_2.$ Denoting by the same symbols the induced functors between the bounded derived categories, we see that the functors $\Phi_1:D^b(\cB)\to D^b(\cE_{\cA,\cB})$ and $\Phi_2:D^b(\cB)\to D^b(\cE_{\cA,\cB})$ are fully faithful and we have a semi-orthogonal decomposition
\begin{equation*}
D^b(\cE_{\cA,\cB})=\la \Phi_1(D^b(\cB)),\Phi_2(D^b(\cB))\ra.
\end{equation*}

Next, we prove \ref{ses}. By Proposition \ref{prop:Serre_quotient} and \cite[Lemma 1.15a]{Ke99} we have a short exact sequence of dg categories
\begin{equation}\label{eq:ses_via_C_AB}
0\to D^b_{\beta(\cB)}(\cC_{\cA,\cB})\to D^b(\cC_{\cA,\cB})\xto{\pi} D^b(\cA)\to 0. 
\end{equation}
Here the first category is the full subcategory of complexes whose homology is contained in the essential image of $\beta.$ By Proposition \ref{prop:abelian_envelope} and Proposition \ref{prop:generalities_on_abelian_envelope} \ref{D^b_is_equiv_toD^b_of_the_envelope} we have an equivalence $D^b(\cE_{\cA,\cB})\simeq D^b(\cC_{\cA,\cB}).$ It induces an equivalence from \eqref{eq:key_ses} to \eqref{eq:ses_via_C_AB}. The assertion about the bounded $t$-structure on $\cT_{\cA,\cB}$ is clear: the composition $\cB\xto{\beta} D^b_{\beta(\cB)}(\cC_{\cA,\cB})\xto{\sim} \cT_{\cA,\cB}$ is given by
\begin{equation*}
X\mapsto \Cone((X,0)\to (X,X)).
\end{equation*}

Now \ref{functors_on_components} is straightforward. Namely, the isomorphisms $q\circ \Phi_1\cong q\circ \Phi_2\cong i$ follow by construction. The restriction of the composition $\Phi_1^L[-1]\circ\Psi$ to the heart $\cT_{\cA,\cB}^{\heartsuit}\simeq \cB$ is given by the natural inclusion $\cB\to D^b(\cB),$ and similarly for $\Phi_2^R\circ\Psi.$ It follows that both $\Phi_1^L[-1]\circ\Psi$ and $\Phi_2^R\circ\Psi$ are left inverses to the realization functor, as required. 
\end{proof}

In the proof of Theorem \ref{th:Quillen} below we avoid using the vanishing of $K_{-1}(\cT_{\cA,\cB})$ (which follows from Theorem \ref{th:key_construction_intro} \ref{ses} and \cite[Theorem 2.35]{AGH19}), because the statement of Theorem \ref{th:Quillen} for $n=0$ is elementary.

\begin{proof}[Proof of Theorem \ref{th:Quillen}] Let $\cA\supset\cB$ be as in Theorem \ref{th:Quillen}. By a simple argument we reduce to the case when each object of $\cA$ is an extension of two objects of $\cB.$ Indeed, for each $n\geq 1$ let $\cA_n\subset\cA$ be the full subcategory which consists of objects with a finite filtration with at most $n$ non-zero subquotients which are contained in $\cB.$ Then each $\cA_n$ is closed under taking subobjects and quotients, $\cA_1=\cB$ and $\bigcup\limits_{n}\cA_n=\cA.$ Since $K$-theory commutes with filtered colimits, we may assume that $\cA=\cA_n$ for some $n.$ Using induction we reduce to the case $\cA=\cA_2.$ We assume this from now on.

To show the isomorphism $K_0(\cB)\xto{\sim} K_0(\cA),$ we define the inverse map $\gamma:K_0(\cA)\to K_0(\cB)$ by putting $\gamma([X])=[Y]+[X/Y]$ for some $Y\subset X$ such that $Y,X/Y\in \cB.$ We only need to check that $\gamma$ is well-defined. To see this, first take another subobject $Y'\subset X$ such that $Y',X/Y'\in\cB.$ Then we have
\begin{multline*}
[Y]+[X/Y]=[Y\cap Y']+[Y/(Y\cap Y')]+[X/Y]=[Y\cap Y']+[X/(Y\cap Y')]\\
=[Y\cap Y']+[Y'/(Y\cap Y')]+[X/Y']=[Y']+[X/Y'].
\end{multline*} 
For a short exact sequence 
\begin{equation*}
0\to X'\xto{f} X\xto{g} X''\to 0
\end{equation*}
in $\cA,$ take some $Y\subset X$ such that $Y,X/Y\in\cB,$ and put $Y'=f^{-1}(Y)\subset X'$ and $Y''=g(Y)\subset X''.$ Then we have
\begin{equation*}
\gamma([X])-\gamma([X'])-\gamma([X''])=([Y]-[Y']-[Y''])+([X/Y]-[X'/Y']-[X''/Y''])=0.
\end{equation*}
Hence, $\gamma$ is well-defined and $K_0(\cB)\cong K_0(\cA).$

To prove the isomorphisms $K_n(\cB)\xto{\sim}K_n(\cA)$ for $n\geq 1$ we use the localization theorem for the short exact sequence \eqref{eq:key_ses}. By Theorem \ref{th:Barwick} the realization functor induces isomorphisms $K_n(\cB)=K_n(D^b(\cB))\xto{\sim} K_n(\cT_{\cA,\cB})$ for $n\geq 0.$ By Theorem \ref{th:key_construction_intro} \ref{functors_on_components}, for $n\geq 0$ the composition
\begin{equation*}
K_n(\cB)\cong K_n(\cT_{\cA,\cB})\to K_n(D^b(\cE_{\cA,\cB}))\cong K_n(\cB)\oplus K_n(\cB)
\end{equation*}
is given by $(-\id,\id).$ Therefore, the long exact sequence of $K$-groups for \eqref{eq:key_ses} implies the isomorphisms $K_n(\cB)\xto{\sim} K_n(\cA)$ for $n\geq 1.$
\end{proof}

\end{document}